\documentclass[12pt,oneside,reqno]{amsart}
\usepackage{latexsym,amsmath,amsfonts,amsthm}
\usepackage[usenames]{color}
\usepackage[usenames,dvipsnames,svgnames,table]{xcolor}
\usepackage[all]{xy}  
\usepackage{upgreek}

\usepackage{mathptmx}

\usepackage{bookman}
\usepackage{relsize} 

\usepackage[]{graphicx}
\usepackage[font=small,labelfont=bf]{caption}
\usepackage{sidecap}
\usepackage{float}

\input xy
\xyoption{all}

\usepackage{hyperref}

\usepackage{amssymb}
\usepackage{amscd}
\usepackage{array}
\usepackage{verbatim}

\theoremstyle{definition}
\newtheorem{thm}{Theorem}[section]
\newtheorem{lem}[thm]{Lemma}

\newtheorem{dfn}[thm]{Definition}

\newtheorem{rmk}[thm]{Remark}

\newcommand{\bei}{\begin{itemize}}
\newcommand{\eei}{\end{itemize}}

\newcommand{\QED}{\rule{0.4em}{2ex}}

\def\conj#1{\overline{#1}}

\def\Dinner#1#2{ \langle #1,\, #2 \rangle_{{}_D}}

\def\vartau{\uptau}

\def\ft#1{\widehat{#1}}
\def\T{\bold T}
\def\Top{\bold{Top}}

\def\ccite#1{\textcolor{Red}{\cite{#1}}}

\numberwithin{equation}{section}
\begin{document}

\title[\normalsize K-\lowercase{inductive} S\lowercase{tructure}]
{\Large \rm T\lowercase{he} K-\lowercase{inductive} S\lowercase{tructure} \lowercase{of} \lowercase{the} \\ \vskip5pt N\lowercase{oncommutative} 
F\lowercase{ourier} T\lowercase{ransform} }

\author{S\lowercase{amuel} G.\,W\lowercase{alters}}
\date{May 20, 2017} 
\address{Department of Mathematics and Statistics, University  of Northern B.C., Prince George, B.C. V2N 4Z9, Canada.}
\email[]{\rm Samuel.Walters@unbc.ca}
\dedicatory{\SMALL Dedicated to Canada on her 150\,${}^{th}$ birthday}
%\thanks{\scriptsize \LaTeX\ File: \it KinductivestructureFourier.tex}  
\urladdr{http://hilbert.unbc.ca}

\begin{abstract} 
The noncommutative Fourier transform of the irrational rotation C*-algebra is shown to have a K-inductive structure (at least for a large concrete class of irrational parameters, containing dense $G_\delta$'s). This is a structure for automorphisms that is analogous to Huaxin Lin's notion of tracially AF for C*-algebras, except that it requires more structure from the complementary projection. 
\end{abstract}

\maketitle

\begin{quote}
{\Small \tableofcontents}
\end{quote}

\newpage

%%%%%%%%%%%%%%%%%%%%%%%%%%%%%%%%%%%%%%%%%%%%%%
\textcolor{blue}{\Large \section{\bf Introduction}}

In this paper we prove that the noncommutative Fourier transform of the irrational rotation C*-algebra $A_\theta$ has a K-inductive structure for at least a large  class of irrationals $\theta$ (containing concrete dense $G_\delta$'s) -- see Theorem \ref{mainthm}. Let us explain this.

\medskip

Let $\mathcal B$ denote the collection of C*-algebras (which we regard as building blocks) consisting of matrix algebras, matrix algebras over the unit circle, or finite direct sums of these.  By a $\mathcal B$-type algebra we mean one that is C*-isomorphic to an algebra in the collection $\mathcal B$. For example, the Elliott-Evans structure theorem \ccite{EE} states that the irrational rotation C*-algebra $A_\theta$ can be approximated by unital C*-subalgebras of the form $M_{q}(C(\mathbb T)) \oplus M_{q'}(C(\mathbb T))$, which are in the class $\mathcal B$.

\medskip

\begin{dfn}\label{defnpreind}
Let $\alpha$ be an automorphism of a unital C*-algebra $A$. We say that $\alpha$ is {\it K-inductive} if, for each $\epsilon > 0$ and each finite subset $S$ of $A$, there exist a finite number of $\mathcal B$-type building block C*-subalgebras $B_1,\dots,B_n$ of $A$ with respective unit projections $e_1,\dots, e_n$, such that

\smallskip

(1) $B_j$ and $e_j$ are $\alpha$-invariant for each $j$,

(2) $\|e_jx - xe_j\| < \epsilon$,  $\forall x\in S$ and each $j$,

(3) $e_jxe_j$ is to within distance $\epsilon$ from $B_j$, $\forall x\in S$ and each $j$,

(4) $[e_1] + \dots + [e_n] = [1]$ in $K_0(A^\alpha)$.
\end{dfn}

Here, $A^\alpha$ is the fixed point subalgebra of $A$ (i.e., the C*-orbifold under $\alpha$). Note that the equality of $K$-classes in condition (4) is stronger than simply requiring it to hold in $K_0(A)$. A projection $e$ is a {\it  matrix projection} in $A$ when it is approximately central and is the unit of a subalgebra $B \cong M_n(\mathbb C)$ (for some $n$), and the cut downs $exe$ are close to $B$ for each $x$ in any prescribed finite subset $S\subset A$.

\medskip

The rotation C*-algebra (or noncommutative 2-torus) $A_\theta$ is the universal C*-algebra generated by unitaries $U,V$ enjoying the Heisenberg relation
\[
VU= e^{2\pi i\theta} UV.
\]
The noncommutative Fourier transform (NCFT) of $A_\theta$ is the canonical order four automorphism (or symmetry) $\sigma$ given by 
\begin{equation}
\sigma(U) = V^{-1}, \quad \sigma(V)= U.
\end{equation}
We will simply say `Fourier transform' (and drop the adjective `noncommutative'). \footnote{The connection between this C*-Fourier transform and the classical Fourier transform $\ft f$ is aptly expressed in terms of the C*-inner product equation $\sigma(\Dinner fg) = \Dinner{\ft f}{\ft g}$ as in \ccite{SWcrelles} (but we will not need this fact here).} The Elliott Fourier Transform problem, which is still open, is the problem of determining the inductive limit structure of the Fourier transform of the irrational rotation C*-algebra $A_\theta$ with respect to basic building blocks consisting of finite dimensional algebras and circle algebras. (Or, more generally, in terms of type I C*-subalgebras.) 

The main result of this paper is the following theorem, where $\mathcal G$ is any of the dense $G_\delta$ sets in $(0,1)$ constructed in Section 2.5 below.

\begin{thm} \label{mainthm} ({\bf Structure Theorem})
{\it For each irrational number $\theta$ in the dense G${}_\delta$-set $\mathcal G$, the noncommutative Fourier transform $\sigma$ of $A_\theta$ is a K-inductive automorphism with respect to matrix algebras. More specifically, 
for each $\epsilon > 0$ and each finite subset $S$ of $A_\theta$, there are three $\mathcal B$-type building block matrix C*-subalgebras 
\[
B_1 \cong M_m(\mathbb C), \qquad  B_2 \cong M_n(\mathbb C), \qquad
B_3 = M \oplus \sigma(M) \oplus \sigma^2(M) \oplus \sigma^3(M), 
\qquad M \cong M_\ell(\mathbb C)
\]
(for some integers $\ell,m,n$), with respective unit projections
\[
e_1,\qquad e_2, \qquad f = g+\sigma(g) +\sigma^2(g)+\sigma^3(g)
\]
where  $g$ is the unit projection of $M$ with $g\sigma^j(g) = 0$ (for $j=1,2,3$), such that 

\smallskip

(1) $B_1,B_2,B_3$ and $e_1,e_2$ are $\sigma$-invariant,

(2) $\|e_jx - xe_j\| < \epsilon$ and $\|gx - xg\| < \epsilon$,  $\forall x\in S$ and $j=1,2$,

(3) $e_jxe_j$ and $gxg$ are to within distance $\epsilon$ from $B_j$ and $M$, respectively, $\forall x\in S$ and each $j=1,2$.

\smallskip

\noindent Further, there exist $\sigma$-invariant unitaries $w,z$ in $A_\theta$ satisfying the equation 
\begin{equation}\label{eqnmainthm}
e_1 + we_2w^* + zfz^* = 1.
\end{equation}
}
\end{thm}

\bigskip

This equation \eqref{eqnmainthm} is equivalent to condition (4) in the above definition since the orbifold $A_\theta^\sigma$ has the cancellation property.

\medskip

We may schematically display the K-inductive structure of the NCFT on $A_\theta$ in terms of building blocks as
``$\bullet \oplus \bullet \oplus \circ \oplus \circ \oplus \circ \oplus \circ$" 
where each bullet $\bullet$ represents a $\sigma$-invariant matrix algebra and the open bullets $\circ$ are matrix algebras that are cyclically permuted by the NCFT.

\medskip

The notion of K-inductive is a natural extension of Huaxin Lin's notion of tracially AF for C*-algebras \ccite{HL} to automorphisms. The one difference is that whereas tracially AF means that there are plenty of finite dimensional projections whose complements are equivalent to some projection in a prescribed hereditary C*-subalgebra, in the case of K-inductive the complement is required to have a rather specific structure -- i.e., is required to be invariantly equivalent to other projections of a building block nature. 

\medskip

We believe that a similar result can be proved for any irrational $\theta$. Our choice of the classes $\mathcal G$ of irrationals makes our computations far more accessible by avoiding number theoretic complications (and helps to make the paper shorter).  A similar approach to that presented in this paper would probably also show that the cubic and hexic transforms of $A_\theta$ -- namely the canonical order 3 and 6 automorphisms studied in \ccite{BW}, \ccite{ELPW}, \ccite{SWnuclearphysicsb} -- are K-inductive automorphisms as well, with respect to matrix algebras possibly including circle algebra building blocks.  The hoped-for conclusion, then, is that all the canonical finite order automorphisms (the only orders being 2, 3, 4, and 6) are K-inductive automorphisms for all irrational $\theta$.

%%%%%%%%%%%%%%%%%%%%%%%%%%%%%%%%%%%%%
\textcolor{blue}{\Large \section{\bf The Framework}}

%%%%%%%%%%%%%%%%%%%%
\subsection{Continuous Field of Fourier Transforms}
We write $U_t, V_t$ for the continuous sections of canonical unitaries of the continuous field $\{A_t: 0<t<1\}$ of rotation C*-algebras such that $V_t U_t = e(t) U_tV_t$, where we've used the now common notation
\[
e(t) := e^{2\pi it}.
\]
(The unitaries $U_t, V_t$ generate $A_t$ for each $t$.)
When dealing with a specific irrational rotation algebra $A_\theta$ we often write its unitary generators simply as $U,V$ instead of $U_\theta, V_\theta$. On the field $\{A_t\}$ there is a field of noncommutative Fourier transforms $\sigma_t$ given on the fiber $A_t$ by
\[
\sigma_t(U_t) = V_t^{-1},\quad \sigma_t(V_t) = U_t.
\]
Often we omit the subscript on $\sigma_t$ and simply write $\sigma$ since there will be no risk of confusion.

%%%%%%%%%%%%%%%%%%%%
\subsection{Basic Matrix Approximation}
We will use the following result from \ccite{SWmathscand}, Theorem 1.5 (a result that was originally rooted in \ccite{SWcrelles}). 

\medskip

\begin{thm}\label{thmdecomp} (\ccite{SWmathscand}, Theorem 1.5.)
{\it Let $\theta>0$ be an irrational number and $p/q>0$ be a rational number in reduced form such that $0<q(q\theta-p)<1$.  Let $\sigma$ be the Fourier transform of $A_\theta$.  Then there exists a Fourier invariant smooth projection $e$ in $A_\theta$ of trace $\vartau(e) = q(q\theta-p)$ and a Fourier equivariant isomorphism
\begin{equation}\label{eta}
\eta: eA_\theta e \to M_q \otimes A_{\theta'} \qquad\text{such that} \qquad
\eta\sigma = (\Sigma\otimes\sigma')\eta
\end{equation}
where $\Sigma$ and $\sigma'$ are Fourier transform automorphisms of $M_q$ and $A_{\theta'}$, respectively, given by
\begin{equation}
\Sigma(u) = v, \quad \Sigma(v) = u^*, \quad \sigma'(a) = b^*, \quad \sigma'(b) = a,
\end{equation}
where $M_q=C^*(u,v)$ and $u,v$ are order $q$ unitary matrices with $vu=e(\frac pq)uv$, and $A_{\theta'}$ is generated by unitaries $a,b$ with $ba = e(\theta')ab$, where $\theta' = \frac{c\theta+d}{q\theta-p}$ is an irrational number in the GL$(2,\mathbb Z)$ orbit of $\theta$. (Here, $c,d$ are integers such that $cp+dq=1$.)

Furthermore, given a sequence of rational approximations $p/q$ of $\theta$ such that $0<q(q\theta-p)< \kappa <1$ for some fixed number $\kappa$, the projection $e$ is a matrix projection: $e$ is approximately central and $\eta(eUe), \eta(eVe)$ are close to order $q$ unitary generators of $M_q$ and which is Fourier invariant -- the approximations here go to 0 as $q\to\infty$.
}
\end{thm}

It is easy to see that a similar result to this theorem applies for rational approximations of $\theta$ such that $0<q(p - q\theta) <1$ -- simply by replacing $\theta$ by $1-\theta$ and using the canonical isomorphism $A_{1-\theta} \cong A_\theta$ which canonically intertwines the Fourier transform.

Since $\theta$ will be fixed throughout the paper, we will write $e_q^+$ for the above canonical projection of trace $q^2\theta-pq$, since it has positive label (or `charge'), and write $e_q^-$ for the canonical projection of trace $pq - q^2\theta$ with negative label (see \eqref{ebEminus}). According to the last assertion of Theorem \ref{thmdecomp}, we can have Fourier invariant matrix projections of both these types.

%%%%%%%%%%%%%%%%%%%%%%%%%%
\subsection{Covariant Projections}

In \ccite{SWcontfield} (but also somewhat evident in \ccite{SWcrelles}) we showed that the Fourier invariant projections $e_q^{+}$ of Theorem 2.1 are instances of one and the same continuous field $\mathcal E(t)$ of projections of the continuous field $\{A_t\}$ of rotation algebras such that $\vartau(\mathcal E(t)) = t$. The relation is canonically furnished by equation \eqref{eqEplus} for $e_q^{+}$ and \eqref{ebEminus} for its negative charge counterpart $e_q^{-}$.

Given an irrational number $t$ and integers $n,k$, where $n\not=0$, one has the canonical unital *-morphism
\[
\zeta_{n,t}: A_{n^2t-k} \to A_t, \quad 
\zeta_{n,t}(U_{n^2t-k}) = U_t^n, \quad
\zeta_{n,t}(V_{n^2t-k}) = V_t^n. 
\]
This map clearly intertwines the Fourier transform
\[
\sigma_t \zeta_{n,t} = \zeta_{n,t}\sigma_{n^2t-k}.
\]
For rational approximations $\frac{p}q < \theta$ such that $0<q^2\theta-pq<1$, we have the projection
\begin{equation}\label{eqEplus}
e_q^+ = \zeta_{q,\theta} \mathcal E(q^2\theta-pq) \ \ \in A_\theta
\end{equation}
whose trace is $q^2\theta-pq$. This is what we mean by saying that $e_q^+$ is covariant (that it arises from the projection field $\mathcal E(t)$ in a natural manner). We could also write down negatively charged projections\footnote{In the mathematical physics literature related to string theory and noncommutative geometry, the Connes-Chern number $-b^2$, for a projection of trace $ab-b^2\theta$, is referred to as the ``charge" of the projection (or that of its associated instanton).} in $A_\theta$ defined by
\begin{equation}\label{ebEminus}
e_b^- = \nu \zeta_{b,1-\theta} \mathcal E(ab-b^2\theta)
\end{equation}
of trace $ab - b^2\theta \in (0,1)$, where $\nu$ is the canonical isomorphism 
\begin{equation}\label{nu}
\nu:A_{1-\theta}\to A_\theta, \quad \nu(U_{1-\theta}) = V_\theta, \quad 
\nu(V_{1-\theta}) = U_\theta.
\end{equation}
We will need to use the parity automorphism $\gamma$ of $A_\theta$ defined by
\[
\gamma(U) = -U, \quad \gamma(V) = -V
\]
because it commutes with the Fourier transform\footnote{In fact, $\gamma$ is the only nontrivial of the toral action automorphisms that commutes with the Fourier transform $\sigma$.} and has the nice effect of flipping the signs of two of the topological invariants below (namely, $\psi_{11},\psi_{22}$), while preserving the others.

\medskip

\subsection{Topological Invariants} 

With $U_\theta, V_\theta$ denoting the canonical unitaries satisfying
\[
V_\theta U_\theta = e(\theta) U_\theta V_\theta
\]
and the Fourier transform defined by
\[
\sigma(U_\theta) = V_\theta^{-1},\quad \sigma(V_\theta) = U_\theta
\]
the following are the basic unbounded ``trace" functionals defined on the canonical smooth dense *-subalgebra $A_\theta^\infty$:
\begin{align}\label{unboundedtraces}
\psi_{10}^\theta(U_\theta^mV_\theta^n) &= e(-\tfrac\theta4(m+n)^2)\,\delta^{m-n}_2, &
\psi_{20}^\theta(U_\theta^mV_\theta^n) &= e(-\tfrac\theta2 mn)\,\delta^m_2\delta^n_2, 
\notag \\ 
\psi_{11}^\theta(U_\theta^mV_\theta^n) &= e(-\tfrac\theta4(m+n)^2)\,\delta^{m-n-1}_2, &
\psi_{21}^\theta(U_\theta^mV_\theta^n) &= e(-\tfrac\theta2 mn) \,\delta^{m-1}_2 \delta^{n-1}_2,
 \\
& \ \ & 
\psi_{22}^\theta(U_\theta^mV_\theta^n) &= e(-\tfrac\theta2 mn)\,\delta^{m-n-1}_2,  \notag
\end{align}
where $\delta_a^{b}$ is {\it divisor delta function} defined to be 1 if $a$ divides $b$, and 0 otherwise. These maps were calculated in \ccite{SWChern} and used in \ccite{SWcjm}, \ccite{SWcrelles}, \ccite{SWcontfield}. (Sometime $\theta$ is omitted from the notation $\psi_{jk}^\theta$ when there is no risk of confusion.)

\medskip

The functionals $\psi_{1j}$ are $\sigma$-invariant $\sigma$-traces and $\psi_{2j}$ are $\sigma$-invariant $\sigma^2$-traces.  Recall that if $\alpha$ is an automorphism of an algebra $A$ (usually a pre-C*-algebra like $A_\theta^\infty$), by an $\alpha$-trace we understand a complex-valued linear map $\psi$ defined on $A$ satisfying the condition
\[
\psi(xy) = \psi(\alpha(y)x)
\]
for each $x,y$ in $A$; and we say that $\psi$ is $\sigma$-invariant when $\psi \sigma = \psi$.  (Clearly, a $\sigma$-trace is automatically $\sigma$-invariant if its domain contains the identity, but a $\sigma^2$-trace need not be $\sigma$-invariant.)  These unbounded linear functionals induce trace maps on the smooth C*-orbifold $A_\theta^{\infty, \sigma} = A_\theta^\infty \cap A_\theta^\sigma$, thereby inducing homomorphisms on $K$-theory $\psi_*: K_0(A_\theta^\sigma) \to \mathbb C$.

\medskip

In \ccite{SWChern} it was shown that $\{\psi_{10}, \psi_{11}\}$ is a basis for the 2-dimensional vector space of all $\sigma$-traces on $A_\theta^\infty$, and that  $\{\psi_{20}, \psi_{21}, \psi_{22}\}$ is a basis for the 3-dimensional vector space of all $\sigma$-invariant $\sigma^2$-traces on $A_\theta^\infty$.

\medskip

The unbounded traces $\psi_{jk}$ along with the canonical bounded trace $\vartau$ comprise the associated Connes-Chern character group homomorphism
for the fixed point algebra $A_\theta^\sigma$:
\[
\T: K_0(A_\theta^\sigma) \to \mathbb C^6, \qquad
\T(x) = (\vartau(x); \ \psi_{10}(x), \psi_{11}(x);\ \psi_{20}(x), \psi_{21}(x), \psi_{22}(x)).
\]
For the identity one has $\T(1) = (1; 1,0; 1, 0, 0)$. It will be convenient to write
\[
\T(x) = (\vartau(x); \Top(x))
\]
where
\[
\Top(x) := (\psi_{10}(x), \psi_{11}(x);\ \psi_{20}(x), \psi_{21}(x), \psi_{22}(x))
\]
consists of the discrete topological invariants of $x$. Indeed, in view of \ccite{SWChern} and \ccite{SWcjm}, the values of the unbounded traces on projections, and on $K_0(A_\theta^\sigma)$, are quantized, with $\psi_{10}, \psi_{11}$ having range in the lattice subgroup $\mathbb Z + \mathbb Z(\frac{1-i}{2})$ of $\mathbb C$; $\psi_{20}, \psi_{21}$ have range $\frac12\mathbb Z$; and $\psi_{22}$ has range $\mathbb Z$. (Cf. Lemma \ref{topinvariantsE} below which gives the topological invariants for the field $\mathcal E(t)$.)

It is straightforward to check that the parity automorphism $\gamma$ changes the signs of $\psi_{11}$ and $\psi_{22}$:
\[
\psi_{11} \gamma = - \psi_{11}, \quad \psi_{22} \gamma = - \psi_{22}
\]
and it keeps the other $\psi_{jk}$ unchanged. Thus, 
\[
\Top(\gamma(x)) = (\psi_{10}(x), -\psi_{11}(x);\ \psi_{20}(x), \psi_{21}(x), -\psi_{22}(x)).
\]

The Connes-Chern map $\T$ was shown to be injective \ccite{SWcjm} for a dense $G_\delta$ set of irrationals $\theta$, but since $K_0(A_\theta^\sigma) \cong \mathbb Z^9$ for all $\theta$ by \ccite{ELPW} or \ccite{AP}, $\T$ is injective for all irrational $\theta$.  This allows us to conclude that since $A_\theta^\sigma$ has the cancellation property for any irrational $\theta$, two projections $e$ and $e'$ in $A_\theta^\sigma$ are unitarily equivalent by a $\sigma$-invariant unitary if and only if $\T(e)=\T(e')$. 

\smallskip

\begin{dfn}
A projection $f$ is called {\it flat} (or {\it $\sigma$-flat}), when it is an orthogonal sum of the form
\[
f = g + \sigma(g) + \sigma^2(g) + \sigma^3(g) 
\]
for some projection $g$. We call such projection $g$ a {\it cyclic} subprojection for $f$ since it is orthogonal to its orbit under $\sigma$ and its orbit sum gives $f$. 
\end{dfn}

Another reason we call $f$ ``flat" is because its topological invariants vanish:
\[
\Top(f) = (0, 0; 0, 0, 0).
\]
Indeed, if $\psi$ is any of the two kinds of unbounded traces in \eqref{unboundedtraces}, we have (can assume $g$ is smooth) $\psi(g) = \psi(gg) = \psi(\sigma^j(g)g) = 0$ for $j=1,2$, hence $\psi(f) =0$.

\smallskip

In \ccite{SWcontfield} we proved that the topological invariants of the continuous section $\mathcal E(t)$ mentioned in Section 2.3 are given as follows.

\begin{lem}\label{topinvariantsE} (\ccite{SWcontfield}, Theorem 1.7.)
{\it The topological invariants of the projection section $\mathcal E(t)$ are
\[
\Top(\mathcal E(t)) = \left(\frac{1-i}2, \frac{1-i}2; \frac12, \frac12, 1\right)
\]
and its trace is $\vartau(\mathcal E(t)) = t$ for each $t\in(0,1)$.
}
\end{lem}

This will allow us to calculate the topological invariants of the canonical projections $e_q^+$ and $e_q^-$ of Theorem \ref{thmdecomp} using the following lemma.

\begin{lem}\label{psizeta} (\ccite{SWcontfield}, Lemma 3.2.)
{\it Let $\theta_n = n^2\theta - k$ where $k,n$ are integers. The unbounded traces on $A_{\theta_n}$ and $A_\theta$ are related by $\zeta_{n,\theta}$ according to the  equations
\begin{align*}
\psi_{10}^\theta \zeta_{n,\theta} &= \psi_{10}^{\theta_n} + i^{-k}\delta_2^n \psi_{11}^{\theta_n},
 & 
\psi_{20}^\theta \zeta_{n,\theta} &=  \psi_{20}^{\theta_n} + (-1)^{k} \delta_2^{n}\psi_{21}^{\theta_n} + \delta_2^{n} \psi_{22}^{\theta_n},
\\
\psi_{11}^\theta \zeta_{n,\theta} &= i^{-k}\delta_2^{n-1} \psi_{11}^{\theta_n},
& 
\psi_{21}^\theta \zeta_{n,\theta} &=  (-1)^{k} \delta_2^{n-1} \psi_{21}^{\theta_n},
\\
 &  & \psi_{22}^\theta \zeta_{n,\theta} &=  \delta_2^{n-1} \psi_{22}^{\theta_n}.
\end{align*}
}
\end{lem}

\medskip

(Lemma \ref{psizeta} is in fact easy to check by directly working out both sides on generic unitaries  $U_{\theta_n}^mV_{\theta_n}^n$.)

\medskip

Combining these two lemmas and applying them to the canonical projection $e_q^+ = \zeta_{q,\theta} \mathcal E(q^2\theta-pq)$, we obtain its topological invariants
\begin{equation}\label{eqplus}
\Top(e_q^+) = 
\left(\tfrac{1-i}2 [1 + (-1)^{q/2}\delta_2^q], \ \ (\tfrac{1-i}2) i^{-pq} \delta_2^{q-1}; \ \ \
\tfrac12 + \tfrac32\delta_2^q,\ \ \tfrac12 (-1)^p \delta_2^{q-1},\ \ \delta_2^{q-1}\right).
\end{equation}
Likewise, the topological invariants of the negatively charged canonical projection $e_b^-$ of trace $ab-b^2\theta$ (where $a,b$ are coprime) are
\begin{equation}\label{esminus}
\Top(e_b^-) = \left(\tfrac{1+i}2[1+(-1)^{b/2}\delta_2^b], \ \  
(\tfrac{1+i}2) i^{-ab}\delta_2^{b-1}; \ \ \
\tfrac12 + \tfrac32\delta_2^{b},\ \ \tfrac12(-1)^{a}\delta_2^{b-1}, \ \  
\delta_2^{b-1}\right).
\end{equation}
To check the latter invariants of $e_b^-$, one uses the following relations between the unbounded traces and the canonical isomorphism $\nu$ of \eqref{nu}:
\[
\psi_{10}^\theta\nu = (\psi_{10}^{1-\theta})^*, \qquad
\psi_{11}^\theta\nu = -i (\psi_{11}^{1-\theta})^*
\]
\[
\psi_{20}^\theta\nu = \psi_{20}^{1-\theta}, \qquad 
\psi_{21}^\theta\nu = - \psi_{21}^{1-\theta}, \qquad
\psi_{22}^\theta\nu =  \psi_{22}^{1-\theta}
\]
which are straightforward to check by working them out on the unitary elements $U_{1-\theta}^mV_{1-\theta}^n$. (Here, $\psi^*$ is the Hermitian adjoint $\psi^*(x) = \conj{\psi(x^*)}$.)

\medskip

To verify the first component of $\Top(e_b^-)$, for example, we compute:
\[
\psi_{10}^\theta(e_b^-) 
= \psi_{10}^\theta \nu \zeta_{b,1-\theta} \mathcal E(ab-b^2\theta) 
= \conj{ \psi_{10}^{1-\theta}\zeta_{b,1-\theta} \mathcal E(\tilde\theta) }
\]
where we have written $\tilde\theta := ab-b^2\theta = b^2(1-\theta) - (b^2-ab)$. By Lemma \ref{psizeta}, with $\theta$ replaced by $1-\theta$, and by Lemma \ref{topinvariantsE} we get 
\begin{align*}
\psi_{10}^{1-\theta} \zeta_{b,1-\theta} \mathcal E(\tilde\theta) 
&= \psi_{10}^{\tilde\theta}\mathcal E(\tilde\theta) + i
^{-(b^2-ab)}\delta_2^b \psi_{11}^{\tilde\theta} \mathcal E(\tilde\theta),
\\
&= \tfrac{1-i}2 + i^{-(b^2-ab)}\delta_2^b \tfrac{1-i}2
\\
&= \tfrac{1-i}2(1 + i^{ab}\delta_2^b)
\\
&= \tfrac{1-i}2 (1 + (-1)^{b/2}\delta_2^b)
\end{align*}
since if $b$ is odd the delta term vanishes and when $b$ is even, $a$ has to be odd so can be removed in the power of $-1$. After conjugating we obtain $\psi_{10}^\theta(e_b^-)$, as asserted. The other invariants of $e_b^-$ in \eqref{esminus} are similarly checked.

\medskip

\subsection{Class of Irrationals $\mathcal G$} 
We begin with any dense set $D$ of rational numbers $\frac{k}m$ in the open interval $(0,\frac12)$ where $k,m\ge1$. We form the following integers
\[
n = 4mk+1, \quad q = n^2, \quad s = n^2+4m^2, \quad p = 4k^2(2n+1), 
\quad r = p + 2n - 3
\]
which can easily be checked to satisfy the modular equation
\begin{equation}\label{psqr}
ps - qr = 1
\end{equation}
(for any $k,m$). Let $\kappa_2, \kappa_1$ be any fixed pair of positive numbers such that 
\begin{equation}\label{kappas}
0 < \kappa_2 \le \frac12 < \kappa_1 < 1, \qquad 1 < \kappa_1 + \kappa_2.
\end{equation}
One checks (using \eqref{psqr}) directly that the following inequality
\begin{equation}\label{preineq}
\frac{r}{s} < \frac{pq-\kappa_1}{q^2} < \frac{rs+\kappa_2}{s^2} < \frac{p}{q}
\end{equation}
holds for large enough $k$. The left inequality holds for large enough $k$ (specifically for $k$ such that $\frac{q}s > \kappa_1$, since $\frac{q}s\to1$ as $k\to\infty$). (Indeed, the left inequality holds for $k$ such that $4k^2 \ge \frac{\kappa_1}{1-\kappa_1}$.)  The middle inequality holds for all $k,m$ by virtue of \eqref{kappas}\footnote{The middle inequality yields the quadratic inequality $\kappa_2 x^2 - x + \kappa_1 > 0$, where $x = q/s$.  By \eqref{kappas}, the quadratic is a decreasing function over the interval $[0,1]$ and is positive at the endpoints, so it is positive on $[0,1]$.}, and the right inequality holds always since $\kappa_2 \le \frac12$.

\medskip

It is easy to see that the difference $\frac{p}q - \frac{2k}m$ goes to $0$ for large $k,m$, hence the set of rationals $\{\frac{p}q: k,m\ge1\}$ is dense in the open interval $(0,1)$, as also is the set $\{\frac{r}s\}$.

\medskip

We can extend slightly inequality \eqref{preineq} to the following
\begin{equation}\label{grandineq}
\frac{2km-\tfrac12}{m^2} < \frac{r}{s} < \frac{pq-\kappa_1}{q^2} 
< \theta < 
\frac{rs+\kappa_2}{s^2} < \frac{p}{q} < \frac{2k}{m}
\end{equation}
where $\theta$ will be the type of irrational that we'll be interested in. The leftmost and rightmost inequalities here can be checked to hold for all $k,m$ since they follow from the equalities
\[
rm^2 - s( 2km - \tfrac12) \ =\ 4k^2m^2 + m^2 + 2km + \tfrac12, \qquad
2kq - pm \ =\ 4k^2m + 2k.
\]
Of course, the remaining inequalities in \eqref{grandineq} hold for large enough $k,m$ depending on choice of $\kappa_1,\kappa_2$ satisfying \eqref{kappas}.

\medskip

The above leads to the construction of various dense $G_\delta$ sets of irrational numbers $\theta$ in $(0,1)$ for each choice of $\kappa_1,\kappa_2$ satisfying \eqref{kappas} and choice of dense set $D$ of rational numbers in $(0,\frac12)$. Such irrationals $\theta$ possess infinitely many pairs of integers $k,m$, and associated rational approximations $\frac{2k}m, \frac{p}q, \frac{r}s$ satisfying \eqref{grandineq}. For example, based on the inner inequality in \eqref{grandineq}, one takes a countable intersection of the dense unions of open intervals
\begin{equation}\label{G}
\mathcal G = \mathcal G_{\kappa_1, \kappa_2} \ = \ \bigcap_{N\ge1} \ \  
\bigcup_{\substack{k,m \ge N \\ \frac{k}m \in D}} 
\left(\frac{pq-\kappa_1}{q^2}, \frac{rs+\kappa_2}{s^2}\right).
\end{equation}
One could conceivably construct specific irrationals in the class $\mathcal G$.

%%%%%%%%%%%%%%%%%%%%%%%%%%%%%%%%%%%%%%%%%%%%%%
\textcolor{blue}{\Large \section{\bf Proof of Structure Theorem}}

We begin the proof with the following lemma. If $B$ is a C*-subalgebra of $A$ and $x\in A$, we use the standard notation $d(x,B)$ for the norm distance between $x$ and $B$: $d(x,B) = \inf\{ \|x-y\|: y\in B\}$.

\medskip

\begin{lem}\label{supporttheorem}  
{\it Let $\theta>0$ be an irrational number and $M,N$ positive coprime integers such that $0 < N(N\theta-M) <1$.  Then for each \ $t \in (0,1) \cap (\mathbb Z + \mathbb Z\theta)$ such that
\[
t<\frac14(N\theta-M)
\]
there exists a cyclic projection $h$ (i.e., $h\sigma^j(h)=0$ for $j=1,2,3$) of trace
\[
\vartau(h) = Nt.
\]
If, in addition, there is a sequence of rationals $M/N$ such that $0 < N(N\theta-M) < \kappa<1$ for some fixed $\kappa$, then for each $\epsilon>0$, there are $N,M$ large enough such that
\begin{enumerate}
\item $\|hU-Uh]\|< \epsilon, \ \|hV-Vh\|< \epsilon$,
\item there is a matrix C*-subalgebra $\frak M$ of $A_\theta$ having $h$ has its unit such that 
\[
d(hUh, \frak M) < \epsilon, \quad d(hVh, \frak M) < \epsilon.
\]
\end{enumerate}
}
\end{lem}
\begin{proof}
Consider the canonical Fourier invariant projection $e$ in $A_\theta$ of trace $\vartau(e) = N(N\theta-M)$ given by Theorem \ref{thmdecomp} and corresponding isomorphism
\[
\eta: eA_\theta e \to M_N \otimes A_{\theta'}, \qquad \text{where} \ \ \ 
\theta' = \frac{c\theta+d}{N\theta-M} %= \bmatrix c & d \\ N & -M \endbmatrix \theta
\]
and $c,d$ are integers such that $cM+dN=1$. %(the matrix here being in GL$(2,\mathbb Z)$).  
Write $t = m + n\theta$, for some integers $m,n$, and let $K = Mn+Nm$ and $L = dn-cm$.  Then 
\[
K\theta' + L = \frac{(Mn+Nm)(c\theta+d) + (dn-cm)(N\theta-M)}{N\theta-M} = \frac{t}{N\theta-M} < \frac14
\]
so that $K\theta' + L$ is in $(0,\frac14) \cap (\mathbb Z + \mathbb Z\theta')$.  By Theorem 1.6 of \ccite{SWmathscand}\footnote{One could also use Theorem 1.5 in \ccite{SWsemiflat}.}, there exists a $\sigma'$-cyclic projection $h'$ in $A_{\theta'}$ of trace $K\theta'+L = \frac{t}{N\theta-M}$, where $\sigma'$ is the Fourier transform of $A_{\theta'}$ in Theorem \ref{thmdecomp}. This gives the cyclic projection
\[
h := \eta^{-1}(I_N\otimes h')
\]
of trace
\[
\vartau(h) = N(N\theta-M) \cdot \frac{t}{N\theta-M} = Nt.
\]
Since the isomorphism $\eta$ is Fourier covariant, as expressed by \eqref{eta}, the projection $h$ is a cyclic subprojection of $e$.

To prove the second assertion of the lemma, assume we have an infinite sequence of rationals $M/N$ such that $0 < N(N\theta-M) < \kappa<1$ for some fixed $\kappa$. In view of the second part of Theorem \ref{thmdecomp}, given $\epsilon>0$ there is $N$ large enough so that $\eta(eUe)$ and $\eta(eVe)$ are to within $\epsilon$ of some elements of the matrix algebra $M_N$.  Then
\[
\frak M := \eta^{-1}(M_N \otimes h') = h\eta^{-1}(M_N)h
\]
is a matrix C*-subalgebra of $A_\theta$ with identity element $h$. (So the algebra $\frak M$ is cyclic under $\sigma$.) As $\eta(h) = I_N\otimes h'$ commutes with $M_N \otimes h'$, the cut downs $hUh = heUeh$ and $hVh = heVeh$ are to within $\epsilon$ of elements of $\frak M$, hence condition (2) holds, and $\|ex-xe\| < \epsilon$, for $x=U,V$. To see that $h$ is approximately central, let $x=U,V$ and write
\[
hx-xh = h(ex - xe) + (ex - xe)h + h exe - exe h
\]
so that from $\|ex-xe\| < \epsilon$ one gets $\|hx-xh\| < 2\epsilon + \|hexe - exeh\|$. Further, since $\eta$ is an isometry we get
\[
\|hexe - exeh\| = \|\eta(h) \eta(exe) - \eta(exe) \eta(h)\|
\]
and since $\eta(exe)$ is to within $\epsilon$ of an element of $M_N \otimes 1$, with which $\eta(h)$ commutes, one gets $\|hexe - exeh\| < 2\epsilon$. Therefore, $\|hx-xh\| < 4\epsilon$ and $h$ is approximately central.
\end{proof}

\medskip

\begin{rmk}
We point out that the proof of this lemma can be modified slightly to give approximately central Fourier invariant projections $h$ of trace $Nt$ (with the $1/4$ factor removed from the hypothesis on $t$).
\end{rmk}

We now have the groundwork necessary in order to proceed with the proof of Theorem \ref{mainthm}.

\medskip

Fix an irrational $\theta$ in the class $\mathcal G$ given by \eqref{G}. 

The inequalities \eqref{grandineq} give three rational convergents of $\theta$ and three respective numbers
\[
0< 2km - m^2\theta < \tfrac12,  \qquad
0< pq - q^2\theta < \kappa_1, \qquad 0 < s^2\theta - rs < \kappa_2. 
\]
We are interested in the following approximately central canonical matrix projections 
\[
e_m^-, \qquad e_q^-,\qquad e_s^+
\]
with respective traces $2km - m^2\theta, \ pq - q^2\theta, \ s^2\theta - rs$. From \eqref{eqplus} and \eqref{esminus} we obtain the topological invariants of the last two to be
\begin{align}
\Top(e_q^{-}) &= (\tfrac{1+i}2,\ \tfrac{1+i}2 i^{-pq};\ \tfrac12,\ \tfrac12(-1)^{p},\  1)
\\
\Top(e_s^{+}) &= (\tfrac{1-i}2,\ \tfrac{1-i}2 i^{-rs};\ \tfrac12,\ \tfrac12(-1)^{r},\ 1)
\end{align}
as $q$ and $s$ are odd. Since $p\equiv 0$ mod $4$ and $rs\equiv -1$ mod $4$ (see first paragraph of Section 2.5), these become
\begin{align}
\Top(e_q^{-}) &= (\tfrac{1+i}2,\ \tfrac{1+i}2;\ \tfrac12,\ \tfrac12,\  1)
\\
\Top(e_s^{+}) &= (\tfrac{1-i}2,\ \tfrac{1+i}2;\ \tfrac12,\ -\tfrac12,\ 1).
\end{align}
Taking the parity $\gamma$ of $e_s^+$ gives
\begin{align}
\Top(\gamma e_s^{+}) &= (\tfrac{1-i}2,\ -(\tfrac{1+i}2);\ \tfrac12,\ -\tfrac12,\ -1)
\end{align}
and adding gives
\[
\Top(e_q^{-}) + \Top(\gamma e_s^{+}) = (1, 0; 1, 0, 0) = \Top(1).
\]
Therefore
\begin{equation}\label{Tofeqes}
\T(1) - \T(e_q^{-}) - \T(\gamma e_s^{+}) = (\vartau_0; 0, 0; 0, 0, 0)
\end{equation}
where the trace value $\vartau_0$ here is
\[
\vartau_0 = 1 - (pq-q^2\theta) - (s^2\theta - rs) = (1+rs-pq) - (s^2-q^2)\theta.
\]
Computing these in terms of the parameters $k,m$, one gets
\[
s^2-q^2 = 8m^2(2m^2+n^2) = 8m^2(16k^2m^2+8km+2m^2+1) 
= 4m^2B
\]
and
\[
1+rs-pq = 4m^2(64k^3 m + 8km +24k^2 - 1) = 4m^2 A
\]
where
\[
A = 64k^3 m+ 8km + 24k^2-1, \qquad B = 2(16k^2m^2+8km+2m^2+1).
\]
Thus, we can write 
\[
\vartau_0 = 4m^2 (A - B\theta).
\]
We now claim that $\vartau_0$ is the trace of an approximately central flat projection
\[
f = g + \sigma(g) + \sigma^2(g) + \sigma^3(g)
\]
whose cyclic subprojection $g$ is approximately central as well. 
First, it is straightforward to check that
\[
2ks-m(r+4) = 4k^2m+2k-3m > 0
\]
is positive (for all $k,m\ge1$), and that one has the equality
\[
sA - Br = 1.
\]
These give the inequality 
\[
\frac{4mA-2k}{4mB-m} < \frac{r}s < \theta
\]
from which we get  
\begin{equation}\label{t-hyp}
t := m(A-B\theta) < \frac14(2k-m\theta) < 1.
\end{equation}
To be sure that $A-B\theta > 0$, in view of \eqref{grandineq} it is enough to see that 
\[
\theta < \frac{rs+\kappa_2}{s^2} < \frac{A}{B}.
\]
Cross multiplying the last inequality here reduces it to $\kappa_2 < \frac{s}B$ (again using $sA - Br = 1$) which holds since $\kappa_2 \le \frac12$ and $\frac12 < \frac{s}B$ follows from $2s = B + 4m^2$.

\medskip

To establish the claim just made, apply Lemma \ref{supporttheorem} with
\[
N(N(1-\theta)-M) = m(2k-m\theta) = \vartau(e_m^-)
\]
i.e., with $N=m$ and $M=m-2k$, and with $t = m(A-B\theta)$. The hypothesis of this lemma that $t < \frac14(2k-m\theta)$ has already been checked in \eqref{t-hyp}.  Therefore, by Lemma \ref{supporttheorem} there exists an approximately central cyclic projection $g$ of trace
\[
\vartau(g) = mt = m^2(A-B\theta).
\]
The second part of Lemma \ref{supporttheorem} (where the ``$\kappa$" there can be taken to be $\frac12$ in view of the inequalities \eqref{grandineq} relating $\theta$ and $2k/m$) gives the matrix cut down approximation for $g$. The corresponding flat projection is then
\[
f = g + \sigma(g) + \sigma^2(g) + \sigma^3(g)
\]
with trace 
\[
\vartau(f) = 4m^2(A-B\theta) = \vartau_0.
\]
Therefore \eqref{Tofeqes} becomes
\[
\T(e_q^{-}) + \T(\gamma e_s^{+}) + \T(f) = \T(1).
\]
where all the underlying projections $e_q^{-}, e_s^{+}, g, f$ are approximately central matrix projections. Since the Connes-Chern map $\T$ is injective we get the following equality of classes in $K_0(A_\theta^\sigma)$
\[
[e_q^{-}] + [\gamma e_s^{+}] + [f] = [1]
\]
as required by Definition 1.1. Since the orbifold C*-algebra $A_\theta^\sigma$ has the cancellation property, this equation of $K$-classes gives equation \eqref{eqnmainthm} of Theorem \ref{mainthm} for some Fourier invariant unitaries $w,z$ -- namely, 
$e_q^{-} + w\gamma e_s^{+}w^* + zfz^* = 1$.

\medskip

This completes the proof of Theorem \ref{mainthm} that the Fourier transform $\sigma$ is K-inductive on the irrational rotation C*-algebra $A_\theta$.

%%%%%%%%%%%%%%%%%%%%%%%%%%%%%%%%%%%%%%%
%%%%%%%%%%%%%%%%%    REFERENCES      %%%%%%%%%%%
%%%%%%%%%%%%%%%%%    REFERENCES      %%%%%%%%%%%
%%%%%%%%%%%%%%%%%    REFERENCES      %%%%%%%%%%%


\bigskip

\begin{thebibliography}{10}

\small

\bibitem{BW}
J. Buck and S. Walters, \emph{Connes-Chern characters of hexic and cubic modules}, J. Operator Theory {\bf57} (2007), 35--65.

\bibitem{ELPW}
S. Echterhoff, W. L\"uck, N. C. Phillips, S. Walters,
\emph{The structure of crossed products of irrational rotation algebras by 
finite subgroups of ${\mathrm SL}_2(\mathbb Z)$}, J. Reine Angew. Math. (Crelle's Journal) {\bf 639} (2010), 173-221.

\bibitem{EE}
G. Elliott and D. Evans, \emph{The structure of the irrational rotation C*-algebra},
Ann. Math. {\bf138} (1993), 477--501.

\bibitem{HL}
H. Lin, \emph{Classification of Simple Tracially AF C*-Algebras}, Canad. J. Math. {\bf53}, No. 1, (2001), 161Ð194.


\bibitem{AP}
A. Polishchuk, \emph{Holomorphic bundles on 2-dimensional noncommutative toric orbifolds}, in \emph{Noncommutative Geometry and Number Theory}, by C. Consani and M. Marcolli (eds.), Vieweg Publ. (2006), 341-359.


\bibitem{SWChern}
S.~Walters, \emph{Chern characters of Fourier modules}, Canad. J. Math. {\bf 52} (2000), No. 3, 633--672.

\bibitem{SWcjm}
S. Walters, \emph{K-theory of non commutative spheres arising from the Fourier automorphism}, Canad. J. Math. {\bf53}, No. 3 (2001), 631--672.


\bibitem{SWcrelles}
S.\,Walters,\! \emph{The AF structure of non commutative toroidal $\mathbb Z/4\mathbb Z$ orbifolds}, J. Reine Angew. Math. (Crelle's Journal) {\bf568} (2004), 139--196. arXiv: math.OA/0207239.

\bibitem{SWmathscand}
S. Walters, \emph{Decomposable Projections Related to the Fourier and flip 
Automorphisms}, Math. Scand. {\bf107} (2010), 174-197.

\bibitem{SWnuclearphysicsb}
S. Walters, \emph{Toroidal Orbifolds of $\mathbb Z_3$ and $\mathbb Z_6$ Symmetries of Noncommutative Tori}, Nuclear Physics B {\bf894} (2015), 496-526. DOI: http://dx.doi.org/10.1016/j.nuclphysb.2015.03.008 

\bibitem{SWcontfield}
S. Walters, \emph{Continuous Fields Of Projections And Orthogonality Relations}, J. Operator Theory {\bf77:1} (2017), 191-203. DOI: 10.7900/jot.2016mar19.2130

\bibitem{SWsemiflat}
S. Walters, \emph{Semiflat Orbifold Projections}, preprint (2016), 16 pages.
 

\end{thebibliography}
\end{document}